\documentclass[12pt]{amsart}
\usepackage[margin=1.15in]{geometry}
\usepackage{amssymb, amsmath, setspace}
\usepackage{color}
\usepackage[all,cmtip]{xy}

\theoremstyle{plain}\numberwithin{equation}{section}
\newtheorem{theorem}{Theorem}[section]
\newtheorem{prop}[theorem]{Proposition}

\newtheorem{lm}[theorem]{Lemma}

\newtheorem{cor}[theorem]{Corollary}
\newtheorem{conj}[theorem]{Conjecture}

\theoremstyle{definition}

\newtheorem{defn}[theorem]{Definition}
\newtheorem{rmk}[theorem]{Remark}
\newtheorem{eg}[theorem]{Example}

\newcommand{\ul} \underline

\newcommand{\mtx}{\left[ \begin{matrix}}
\newcommand{\mtxend}{\end{matrix}\right]}

\newcommand{\rank}{\mbox{rank}}
\newcommand{\state}{\mbox{State}}
\newcommand{\conv}{\mbox{conv}}
\newcommand{\vol}{\mbox{vol}}
\newcommand{\A}{\mathcal A}
\newcommand{\G}{\mathcal G}
\newcommand{\U}{\mathcal U}
\newcommand{\C}{\mathcal C}
\newcommand{\<}{\prec}

\begin{document}

%\title[Nov '07.]{Graver bases of rational normal scrolls}
\title[Gr\"obner bases of Scrolls]{On the Universal Gr\"obner bases of varieties of minimal degree}

\author[S. Petrovi\'c]{Sonja Petrovi\'c}
\address{Department of Mathematics, University of Kentucky, Lexington, KY 40506, USA}
\email{{\tt petrovic@ms.uky.edu}}

\begin{abstract}
A universal Gr\"obner basis of an ideal is the union of all its reduced Gr\"obner bases.  It is contained in the Graver basis, the set of all primitive elements.
Obtaining an explicit description of either of these sets, or even a sharp degree bound for their elements, is a nontrivial task.\\
In their '95 paper, Graham, Diaconis and Sturmfels give a nice combinatorial description of the Graver basis for any rational normal curve in terms of primitive
partition identities.  Their result is extended here to rational normal scrolls.  The description of the Graver bases of scrolls
is given in terms of {\em{colored}} partition identities.  This leads to a sharp bound on the degree of Graver basis elements, which is always attained by a circuit.\\
Finally, for any variety obtained from a scroll by a sequence of projections to some of the coordinate hyperplanes, the degree of any element in any reduced
Gr\"obner basis is bounded by the degree of the variety.
\end{abstract}

\maketitle

\subsubsection*{Acknowledgment.} The author would like to thank Bernd Sturmfels for suggesting a generalization of primitive partition identities.

\section{Introduction}

Fix a subset $\A=\{a_1,\dots,a_n\}$ of $\mathbb Z^d$.  The set $\A$ determines a toric ideal in the following way.
Every vector $u\in\mathbb Z^n$ can be written uniquely as $u=u^+-u^-$ where $u^+$ and $u^-$ are nonnegative and have disjoint support.
Considering $\A$ as a $d\times n$ matrix induces a parametrization of a variety $X:=X_\A$ whose defining ideal is the toric ideal $I_\A:=(x^{u^+}-x^{u^-} : Au=0)$ in
the polynomial ring $k[{\bf{x}}]:=k[x_1,\dots,x_n]$.  We may write $I_X$ for $I_\A$.
A standard reference on toric ideals is \cite{St}.

Given any term order $\<$ on the monomials of $k[{\bf{x}}]$, the initial ideal of $I_\A$ is defined to be $in_{\<}(I_\A) := (in_{\<}(f):f\in I_\A)$.
Any generating set $\G_\<$ of the ideal such that $in_{\<}(I_\A) = (in_{\<}(g):g\in\G_\<)$ is called a {\emph{Gr\"obner basis}}.  If $\G_\<$ is {\emph{reduced}}, that is, no term of $g$
is divisible by $in_{\<}(f)$ for any $f,g\in \G_\<$, then $\G_\<$ is uniquely determined by the term order $\<$.
For applications of Gr\"obner bases, see \cite{LeHoTh} and \cite{St}.

 The union of the (finitely many) reduced Gr\"obner bases
of $I_\A$ is called the {\emph{universal Gr\"obner basis}} and denoted $\U_\A$.  It is contained in the set of primitive binomials called the {\emph{Graver basis}} $\G r_\A$ of $I_\A$;
 a binomial $x^{u^+}-x^{u^-}\in I_\A$ is called primitive if there is no $x^{v^+}-x^{v^-}\in I_\A$ such that $x^{v^+}|x^{u^+}$ and $x^{v^-}|x^{u^-}$.
 A set of primitive binomials with minimal support is the set $\C_\A$ of {\emph{circuits}} of the ideal.  It is known that $\C_\A \subset \U_\A \subset \G r_\A$.
 In general, both containments are proper.

There exists a general bound on the degrees of the elements of the universal Gr\"obner basis (\cite{St}), however it is far too large for many specific examples.
One might expect the sharp bound to be smaller for varieties that are special in some sense.

 Rational normal scrolls are examples of varieties of {\emph{minimal degree}}, that is, the varieties which attain the general lower bound $\deg(X) \geq \mbox{codim}(X)+1$.
They have been classified (\cite{EisHa}) as quadratic hypersurfaces, rational normal scrolls, the Veronese surface in $\mathbb P^5$, and cones over these.
 The scrolls are the only infinite family among these.
Their defining ideals have the following special property:
\begin{theorem}[Corollary \ref{degUnivS}]
The degree of any binomial in the Graver basis (and the universal Gr\"obner basis) of any rational normal scroll is bounded above by the degree of the scroll.
\end{theorem}
We also derive a sharp bound (Theorem \ref{degGraverS}), which is usually much smaller then the degree of the scroll.

Another remarkable property of the defining ideals of rational normal scrolls is that their Graver bases admit a particularly nice combinatorial description.
Namely, each primitive element in the ideal of a scroll corresponds to a suitable {\emph{primitive colored partition identity}}.
Such a characterization of primitive elements imposes a restriction on the {\emph{structure}} of any binomial in the universal Gr\"obner bases of scrolls.

The paper is organized as follows.
Section $2$ contains the necessary information about the defining ideals of rational normal scrolls.  In Section $3$ we introduce colored partition
identities and use them to characterize the Graver bases of the scrolls (Proposition \ref{GraverPcpiLm}),
generalizing the result for rational normal curves in \cite{GDS}. %, where the Graver basis is described in terms of primitive partition identities.
Section $4$ contains the degree bounds.  An important consequence of the sharp bound in Theorem \ref{degGraverS} is that if  $X$ is any variety that
can be obtained from a scroll by a sequence of projections to some of the coordinate hyperplanes, then the degree of the variety
 gives an upper bound on the degrees of elements in the universal Gr\"obner basis of its defining ideal $I_X$.
In the final section, we conjecture that the universal Gr\"obner basis equals the Graver basis for any scroll, and discuss its consequences.
  We also derive the dimension of the state polytopes of scrolls. \\

\section{Parametrization of Scrolls}
Let $S:=S(n_1-1,\dots,n_c-1)$ be the rational normal scroll in $\mathbb P^{n_1+\dots +n_c -c-1}$. Its defining ideal is given by $I_2 M$, where
\begin{align*}
M:=[ M_{n_1} | M_{n_2} | \dots | M_{n_c} ],\mbox{  and  } M_{n_j}:=\mtx x_{j,1} &\dots &x_{j,n_j-1} \\ x_{j,2} &\dots &x_{j,n_j}\mtxend.
\end{align*}
If $c=1$, then the $2$-minors of the matrix above give the defining ideal of a rational normal curve $S(n-1)$ in $\mathbb P^{n-1}$ (\cite{EisHa}).

\begin{lm}{\label{parametrization}}
 $I_S=\ker \varphi$, where $\varphi(x_{j,i}) =
[v_1^1,\dots, v_j^1,v_{j+1}^0,\dots, v_c^0, t^i ]^T $ for $1\leq j\leq c$.  That is, the matrix $\A$ that encodes the parametrization of the scroll $S$
is
\begin{align*}\A = \mtx   1  &\dots &1 &1 &\dots &1 &\dots &1 &\dots &1  \\ 0  &\dots &0 &1 &\dots &1 &\dots &1 &\dots &1 \\ \vdots  & & & & & & & & &\vdots
\\ 0 &\dots &0 &0&\dots &0 &\dots &1 &\dots &1 \\ 1 &2 \dots &n_1 &1&\dots &n_2 &\dots &1 &\dots &n_c \mtxend .\end{align*}
\end{lm}
\begin{proof}
Indeed, let the generators of $I_S$ be the minors $$m_{i,j,k,l}:=x_{i,k}x_{j,l+1}-x_{j,l}x_{i,k+1}$$ for $1\leq i,j\leq c$, $1\leq k\leq n_i-1$, and $1\leq l\leq
n_j-1$.  (Note that we allow $i=j$ and $k=l$.) Then the exponent vector of $m_{i,j,k,l}$ is $$v_{i,j,k,l}=[0,\dots, 0, 1, 0,\dots ,0, -1, 0,\dots ,0, -1,
0,\dots, 0, 1, 0,\dots ,0] $$ where the positive entries are in columns $n_1+\dots+n_{i-1}+k$ and $n_1+\dots+n_{j-1}+l+1$, while the negative entries are in
columns $n_1+\dots+n_{j-1}+l$ and $n_1+\dots +n_{i-1}+k+1$.  (If $i=j$ and $k=l$, then the two locations for the negative entries coincide; in that case, the
negative entry is $-2$.)  Denote by $\A_{c}$ the $c^{th}$ column of $\A$.  Then clearly $$\A_{n_1+\dots+n_{i-1}+k}+\A_{n_1+\dots+n_{j-1}+l+1}= \A_{n_1+\dots+n_{j-1}+l}+ \A_{n_1+\dots +n_{i-1}+k+1}$$ since $$ \mtx 1\\ \vdots\\ \vdots\\ 1\\ 0\\ \vdots\\ 0\\k
\mtxend + \mtx 1\\ \vdots\\ 1\\ 0\\ \vdots\\ \vdots\\ 0\\ l+1 \mtxend = \mtx 1\\ \vdots\\ \vdots\\ 1\\ 0\\ \vdots\\ 0\\k+1 \mtxend + \mtx  1\\ \vdots\\ 1\\ 0\\
\vdots\\ \vdots\\ 0\\ l \mtxend .$$ Thus $m_{i,j,k,l}\in I_{\A}$ for each generator $m_{i,j,k,l}$ of $I_S$.

In addition, the matrix $\A$ has full rank; thus the dimension of the variety it parametrizes is $\rank \A -1 = c$.  But this is precisely the
dimension of the scroll $S$.  %%% We have just shown that $I_S$ is the toric ideal $I_\A$ of the variety parametrized by $\A$.
\end{proof}

\begin{eg}
The ideal of the scroll $S(3,2)$ is $I_{\A_{S(3,2)}}$ where
$$\A_{S(3,2)}=\mtx 1&1&1&1&1&1&1\\0&0&0&0&1&1&1\\1&2&3&4&1&2&3 \mtxend.$$
\end{eg}
\vspace{0.1in}

\section{Colored partition identities and Graver bases}

Let us begin by generalizing the definitions from Chapter $6$ of \cite{St}.
\begin{defn}
A {\bf{ colored partition identity}} (or a {\bf{cpi}}) in the colors $(1)$,$\dots$,$(c)$ is an identity of the form
\begin{align*}
\textcolor{blue}{a_{1,1}}+\dots+\textcolor{blue}{a_{1,k_1}}+\textcolor{red}{a_{2,1}}+\dots+\textcolor{red}{a_{2,k_2}}+ & \dots + \textcolor{green}{a_{c,1}}+\dots+\textcolor{green}{a_{c,k_c}} = \\
\textcolor{blue}{b_{1,1}}+\dots+\textcolor{blue}{b_{1,s_1}}+\textcolor{red}{b_{2,1}}+\dots+\textcolor{red}{b_{2,s_2}}+ & \dots + \textcolor{green}{b_{c,1}}+\dots+\textcolor{green}{b_{c,s_c}},
 \tag{*}\label{cpi}
\end{align*}
where $1\leq a_{p,j},b_{p,j} \leq n_p$ are positive integers for all $j$, $1\leq p\leq c$ and some positive integers $n_1$, $\dots$, $n_c$.
\end{defn}

If $c=1$ then this is precisely the definition of the usual partition identity with $n=n_1$.

\begin{rmk}
A cpi in $c$ colors with $n_1$, $\dots$, $n_c$ as above is a partition identity (in one color) with largest part $n=\max\{n_1,\dots ,n_c\}$.
\end{rmk}

\begin{eg}
Denote by $i_r$ the number $i$ colored red, and by $i_b$ the number $i$ colored blue.  Then
\begin{align*} \textcolor{red}{1_r} + \textcolor{red}{4_r} + \textcolor{blue}{3_b} = \textcolor{blue}{5_b} + \textcolor{blue}{1_b} + \textcolor{red}{2_r}
\end{align*}
is a colored partition identity with two colors, with $n_1= 4$ and $n_2=5$.  Erasing the coloring %and adding $4$ to all the blue numbers
gives $1 + 4 + 3 = 5 + 1 + 2$, a (usual) partition identity with largest part $n=5$.
\end{eg}

\begin{defn}
A colored partition identity (\ref{cpi}) is a {\bf{ primitive }} cpi (or a {\bf{ pcpi}}) if there is no proper sub-identity $a_{-,i_1}+\dots+a_{-,i_l} = b_{-,j_1}+\dots
+b_{-,j_t}$, with $1\leq l+t < k_1+\dots+k_c+s_1+\dots+s_c $, which is a cpi.

A cpi is called {\bf{homogeneous}} if $k_1+\dots +k_c = s_1+\dots +s_c$.  If $k_j=s_j$ for $1\leq j\leq c$, then it is called {\bf{color-homogeneous}}.
The {\bf{degree}} of a pcpi is the number of summands $ k_1+\dots +k_c+s_1+\dots +s_c$.
\end{defn}

Note that color-homogeneity implies homogeneity, and that a homogeneous pcpi need not be primitive in the inhomogeneous sense.

\begin{eg}
Here is a list of all primitive color-homogeneous partition identities with $c=2$ colors and $n_1=n_2=3$.
\begin{align*}
1_1 +3_1 &= 2_1 +2_1 \\
1_1 +\textcolor{blue}{2_2} &= 2_1 +\textcolor{blue}{1_2} \\
1_1 +1_1 +\textcolor{blue}{3_2} &= 2_1 +2_1 +\textcolor{blue}{1_2} \\
1_1 +\textcolor{blue}{3_2} &= 2_1 +\textcolor{blue}{2_2}  \\
2_1 +\textcolor{blue}{3_2} &= 3_1 +\textcolor{blue}{2_2} \\
2_1 +\textcolor{blue}{2_2} &= 3_1 +\textcolor{blue}{1_2} \\
1_1 +\textcolor{blue}{3_2} &= 3_1 +\textcolor{blue}{1_2}\\
\textcolor{blue}{1_2} +\textcolor{blue}{3_2} &= \textcolor{blue}{2_2} +\textcolor{blue}{2_2} \\
1_1 +\textcolor{blue}{3_2} +\textcolor{blue}{3_2} &= 3_1 +\textcolor{blue}{2_2} +\textcolor{blue}{2_2} \\
1_1 +\textcolor{blue}{2_2} +\textcolor{blue}{2_2} &= 3_1 +\textcolor{blue}{1_2} +\textcolor{blue}{1_2}\\
2_1 +2_1 +\textcolor{blue}{3_2}&= 3_1 +3_1 +\textcolor{blue}{1_2}
\end{align*}
\end{eg}

We are now ready to relate the ideals of scrolls and the colored partition identities.
\begin{lm}
A binomial $x_{1,a_{1,1}}\dots x_{1,a_{1,k_1}} \dots x_{c,a_{c,1}}\dots x_{c,a_{c,k_c}} - x_{1,b_{1,1}} \dots x_{c,b_{c,s_c}}$ is in the ideal
$I_{\A_{S(n_1-1,\dots,n_c-1)}}$ if an only if (\ref{cpi}) is a color-homogeneous cpi.
\end{lm}
\begin{proof} This follows easily from the definitions and Lemma \ref{parametrization}.
\end{proof}

\begin{eg}
Let $c=2$. Then $$A:=\A_{S(n_1-1,n_2-1)} =\mtx  1&\dots &1 &1&\dots &1\\ 0&\dots &0&1&\dots &1\\1&\dots &n_1 &1&\dots &n_2 \mtxend$$ and $$I_A =I_2\mtx
x_{1,1} & \dots & x_{1,n_1-1} & x_{2,1} &\dots &x_{2,n_2-1} \\ x_{1,2} &\dots &x_{1,n_1} &x_{2,2} &\dots &x_{2,n_2} \mtxend . $$
Then $x_{1,a_{1,1}}\dots x_{1,a_{1,k_1}}x_{2,a_{2,1}}\dots x_{2,a_{2,k_2}} - x_{1,b_{1,1}} \dots x_{1,b_{1,s_1}} x_{2,b_{2,1}} \dots x_{2,b_{2,s_2}}\in I_A$ if
and only if
$$\mtx  v_1^{k_1+k_2} \\v_2^{0+k_2} \\ t^{a_{1,1}+\dots+a_{2,k_2}} \mtxend = \mtx v_1^{s_1+s_2} \\ v_2^{0+s_2} \\ t^{b_{1,1}+\dots +b_{2,s_2}}\mtxend $$
if and only if $k_1+k_2=s_1+s_2$, $k_2=s_2$, and
$$a_{1,1}+\dots+a_{1,k_1}+a_{2,1}+\dots+a_{2,k_2} =  b_{1,1}+\dots+b_{1,s_1}+b_{2,1}+\dots+b_{2,s_2}. $$
The last equality clearly describes a color-homogeneous pcpi.
\end{eg}

The Lemmas above imply the following characterization of the Graver bases of rational normal scrolls.

\begin{prop}{\label{GraverPcpiLm}}
The Graver basis elements for the scroll $S(n_1-1,\dots,n_c-1)$ are precisely the color-homogeneous primitive colored partition identities of the form
(\ref{cpi}).
\end{prop}
\begin{proof} With all the tools in hand, it is not difficult to check that the binomial in the ideal of the scroll is primitive if and only if the corresponding colored partition identity
is primitive.
\end{proof}

If $c=1$, this is just the observation in Chapter $6$ of \cite{St}.

\vspace{0.1in}

\section{Degree bounds}

Now we can generalize the degree bound given in \cite{St} for the rational normal curves.  The degree bound is sharp, and it is remarkable that it is always
attained by a circuit.
By a {\emph{subscroll}} of $S(n_1-1,\dots,n_c-1)$ we mean a scroll $S':=S(n'_1-1,\dots,n'_c-1)$ such that $n'_i\leq n_i$ for each $i$.
Clearly, $I_{S'}$ can be obtained from $I_S$ by eliminating variables.
\begin{theorem}\label{degGraverS}
Let $S:=S(n_1-1,\dots,n_c-1)$ for $c\geq 2$. Let $P$ and $Q$ be the indices such that
$$n_P=\max\{n_i: 1\leq i\leq c\}$$ and $$n_Q=\max\{n_j: 1\leq j\leq c, j\neq P\}.$$  Then the degree of any
primitive binomial in $I_S$ is bounded above by $$n_P+n_Q-2.$$  This bound is sharp exactly when $n_P-1$ and $n_Q-1$ are relatively prime.

More precisely, the primitive binomials in $I_S$ have degree at most $$u+v-2,$$
where $u$ and $v$ are maximal integers such that $S(n'_1-1,\dots,n'_c-1)$ is a subscroll of $S$ with $n'_i=u$ and $n'_j=v$ for some $1\leq i,j \leq c$,
and subject to $(u-1,v-1)=1$.

This degree bound is sharp; there is always a circuit having this degree. For any number of colors $c$, such a maximal degree circuit is two-colored.
\end{theorem}

Before proving the Theorem, let us look at an example.
\begin{eg}
Consider the scroll $S(5,6)$.  Here $n_P-1=6$ and $n_Q-1=5$, and since they are relatively prime, the sharp degree bound is $5+6=11$.
On the other hand, if $S:=S(4,4,2,2)$, then $n_P-1=n_Q-1=4$ so we look for a subscroll $S':=S(4,3,2,2)$. Then $u-1=4$ and $v-1=3$, and the degree of any primitive element
is at most $7$.  Finally, if $S:=S(5,5,5)$, then $n_P-1=n_Q-1=5$. The desired subscroll is $S':=S(5,4,4)$ so that the degree bound is $u-1+v-1=5+4=9$.
\end{eg}

\begin{proof}
Let $x_{1,a_{1,1}}\dots x_{c,a_{c,k_c}} - x_{1,b_{1,1}} \dots x_{c,b_{c,k_c}}\in I_S$. Consider the corresponding color-homogeneous pcpi:
\begin{align*}
a_{1,1}+\dots+a_{1,k_1}+a_{2,1}+\dots+a_{2,k_2}+ & \dots + a_{c,1}+\dots+a_{c,k_c} = \\ b_{1,1}+\dots+b_{1,k_1}+b_{2,1}+\dots+b_{2,k_2}+ & \dots +
b_{c,1}+\dots+b_{c,k_c}.
 \tag{**}\label{hom-pcpi}
\end{align*}
Note that the number of terms on either side of (\ref{hom-pcpi}) equals the degree of the binomial. We shall first show that $k_1+\dots +k_c \leq n_P+n_Q-2$
holds for (\ref{hom-pcpi}).

Let $d_{i,j}=a_{i,j}-b_{i,j}$ be the differences in the $i^{th}$-color entries for $1\leq j\leq k_i$, $1\leq i\leq c$. Then  $$\sum_{\substack{ 1\leq i\leq c\\
1\leq j\leq k_i }}{d_{i,j}=0} .$$ Separating positive and negative terms gives an inhomogeneous pcpi $\sum d_{i,j}^+ = \sum d_{i,j}^-$.  Indeed, if it is not primitive then there would be a
subidentity in (\ref{hom-pcpi}).  Note that an inhomogeneous pcpi is defined to be a ppi with arbitrary coloring. Therefore, the sum-difference
algorithm from the proof of Theorem 6.1. in \cite{St} can be applied.  For completeness, let us recall the algorithm.\\
%\begin{alg}\rm
\phantom{xxx}\texttt{Set }$x:=0,$ $\mathcal P:=\{ d_{i,j}^+ \},$ $\mathcal N:=\{ d_{i,j}^- \}.$\\
\phantom{xxx}\texttt{While } $\mathcal P \cup \mathcal N$ \texttt{ is non-empty do } \\
\phantom{xxxxxx}\texttt{if } $x\geq 0$\\
\phantom{xxxxxxxxx} \texttt{then select an element } $\nu\in\mathcal N$\texttt{, set }$x:=x-\nu$\texttt{ and }$\mathcal N:=\mathcal N\backslash \{\nu\}$\\
\phantom{xxxxxxxxx} \texttt{else select an element } $\pi\in\mathcal P$\texttt{, set }$x:=x+\pi$\texttt{ and }$\mathcal P:=\mathcal P\backslash \{\pi\}$.\\
%\end{alg}
The number of terms in the pcpi is bounded above by the number of values $x$ can obtain during the run of the algorithm. Primitivity ensures no value is reached
twice. Let $$D_{i,+}:=\max_j \{d_{i,j}: d_{i,j}>0 \}$$ and $$D_{i,-}:=\max_j \{-d_{i,j}: d_{i,j}<0 \}.$$   Then  $k_1+\dots +k_c \leq
\max_i \{D_{i,+}\} + \max_i \{D_{i,-} \} =: D_+ + D_-$ (Corollary 6.2 in \cite{St}). Let $D_+$ and $D_-$ occur in colors $P$ and $Q$, respectively, so that $D_+=a_P-b_P$, and $D_-=b_Q-a_Q$.
Then the sequence of inequalities $$ 1+D_+ +1 \leq 1+D_+ +b_P=1+a_P\leq 1+n_P \leq a_Q + n_P= b_Q - D_- + n_P\leq n_Q-D_- +n_P $$ implies that $$
 D_+ + D_- \leq n_P+n_Q-2, $$ and the degree bound follows.

The maximum degree occurs when there is equality in the above sequence of inequalities, and $x$ reaches every possible value during the run of the algorithm.
Following the argument of Sturmfels, this means that the inhomogeneous pcpi $\sum{d_{i,j}}=0$ is of the form $$\underbrace{D_+ + \dots + D_+}_{D_- \mbox{ terms}}
= \underbrace{D_- + \dots D_-}_{D_+ \mbox{ terms}}.$$
 In addition, $$ 1+D_+ +1 =1+D_+ +b_P=1+a_P= 1+n_P = a_Q + n_P= b_Q - D_- + n_P=n_Q-D_- +n_P $$ implies
that $b_P=1$, $a_P=n_P$, $a_Q=1$, and $b_Q=n_Q$.  Therefore, the maximal degree identity $\sum{d_{i,j}}=0$ provides that (\ref{hom-pcpi}) is of the following
form: $$ \underbrace{1_P+\dots +1_P}_{n_Q-1 \mbox{ terms}} + \underbrace{n_Q+\dots +n_Q}_{n_P-1 \mbox{ terms}} = \underbrace{1_Q+\dots +1_Q}_{n_P-1 \mbox{
terms}} + \underbrace{n_P+\dots +n_P}_{n_Q-1 \mbox{ terms}}, $$ where $1_P$ denotes the number $1$ colored using the color $P$.
%%end of eqn
This colored partition identity is primitive if and only if there does not exist a proper subidentity if and only if $n_P-1$ and $n_Q-1$ are relatively prime.
Indeed, if $n_P-1=zy$ and $n_Q-1=zw$ for some $z,y,w\in\mathbb N$, then there is a subidentity of the form
$$ \underbrace{1_P+\dots +1_P}_{w \mbox{ terms}} + \underbrace{n_Q+\dots +n_Q}_{y \mbox{ terms}} = \underbrace{1_Q+\dots +1_Q}_{y \mbox{ terms}} + \underbrace{n_P+\dots +n_P}_{w \mbox{ terms}}.$$
Furthermore, assume that $n_P-1$ and $n_Q-1$ are relatively prime.  Then the exponent vector of the binomial corresponding to the maximal degree identity has
support of  cardinality four. It is thus a circuit for any $c\geq 2$.  Clearly, it is a two-colored circuit, regardless of the number of colors $c$ in
our scroll
 $S$.

Finally, if $n_P-1$ and $n_Q-1$ are not relatively prime, the degree $n_Q+n_P-2$ cannot be attained by a primitive binomial. In that case, we may simply
eliminate one of the variables to obtain a smaller scroll, say $S':=S(n_1-1,\dots, n_P-2, \dots, n_c-1)$, whose defining ideal is embedded in that of $S$
(that is, $u:=n_P-1$ and $v:=n_Q$). Clearly, primitive binomials from $I_{S'}$ lie in $I_S$. If $u-1$ and $v-1$ are relatively prime, then we have the
smaller bound for the degree: $n_P+n_Q-3$. If not, we continue eliminating variables until the condition is satisfied.

This completes the proof.
\end{proof}

\begin{rmk}
In view of the comment on p.36 of \cite{St}, it is interesting to note that in the case of varieties of minimal degree, the maximum degree of any Graver basis element
is attained by a circuit. This is not true in general.
\end{rmk}

Now the following is trivial.
\begin{cor}\label{degUnivS}
The degree of any binomial in the Graver basis (and the universal Gr\"obner basis) of any rational normal scroll is bounded above by the degree of the scroll.
\end{cor}

In addition, this also gives the upper bound for the degrees of any element in the universal Gr\"obner basis of any variety whose parametrization can be embedded
into that of a scroll, generalizing Corollary (6.5) from \cite{St}.
\begin{cor}
Let $X$ be any toric variety that can be obtained from a scroll by a sequence of projections to some of the coordinate hyperplanes.
Then the degree of an element of any reduced Gr\"obner  basis of $I_X$ is at most the degree of the toric variety $X$.
\end{cor}
\begin{proof}The claim follows from degree-preserving coordinate projections and the elimination property of the universal Gr\"obner basis.
 The variety $X=X_\A$ is parametrized by
\begin{align*}
\A = \mtx   1  &\dots &1 &1 &\dots &1 &\dots &1 &\dots &1  \\ 0  &\dots &0 &1 &\dots &1 &\dots &1 &\dots &1 \\ \vdots  & & & & & & & & &\vdots
\\ 0 &\dots &0 &0&\dots &0 &\dots &1 &\dots &1 \\ i_{1,1} &i_{1,2} \dots &i_{1,r_1} & i_{2,1}&\dots &i_{2,r_2} &\dots &i_{c,1} &\dots &i_{c,r_c} \mtxend
\end{align*}
In what follows, we may assume that $1=i_{k,1}< \dots <i_{k,r_k}=n_k$ for $1\leq k\leq c$.  Then $X$ can be obtained by coordinate projections from the
scroll $S:=S(n_1-1,\dots,n_c-1)$, parametrized by $\A_{S}$ as before.
The degree of the toric variety $X_\A$ is the normalized volume of the polytope formed by taking the convex hull of the columns of $\A$. But
$\vol(\conv(\A)) = \vol(\conv(\A_{S}))$ implies that the two varieties have the same degree.

Suppose $x^u-x^v$ is in some reduced Gr\"obner basis of $I_X$. Then Proposition 4.13. and Lemma 4.6. in \cite{St} provide that
$x^u-x^v\in\mathcal U_\A \subset \mathcal U_{\A_S}\subset \mathcal Gr_{\A_S}$.  Applying Corollary \ref{degUnivS} completes the proof.
\end{proof}

\begin{rmk}
In particular, note that this degree bound (which equals the degree of the scroll, $n_1+\dots+n_c-c$) is always better then the general one given for toric ideals in \cite{St},
Corollary 4.15, which equals $1/2 (c+2)(n_1+\dots +n_c-c-1)D(\A)$
where $D(\A)$ is the maximum over all $(c+1)$-minors of $\A$.
\end{rmk}

Let us conclude this section by listing the number of all elements in the Graver basis of some small scrolls, sorted by degree of the binomial.
The entries in this table have been obtained using the software 4ti2 \cite{4ti2}, which was essential in this project.

\begin{tabular}{|c||l|l|l|l|l|l|l|l|l|l|l|}
\hline
%Scroll &\multicolumn{10}{l|}{Number of primitive binomials of each degree:}\\
&\multicolumn{10}{l|}{Degrees}\\
%\cline{2-11}
Scroll  &2\phantom{xxx}&3\phantom{xxx}&4\phantom{xxx}&5\phantom{xxx}&6\phantom{xxx}&7\phantom{xxx}&8\phantom{xxx}&9\phantom{xxx}&10\phantom{xx}&11\phantom{xx}\\
\hline\hline
S(2,2) &7 &4 & & & & & & & & \\\hline
S(2,2,2) &18& 24& & & & & & & & \\\hline
S(4) &7& 7& 2& & & & & & & \\\hline
S(3,2) &12& 16& 4& 1& & & & & & \\\hline
S(3,2,2) &26& 58& 22& 4& & & & & & \\\hline
S(3,3) &20& 40& 18& 4& & & & & & \\\hline
S(3,3,2,2) &59& 242& 208& 36& & & & & & \\\hline
S(4,2) &19& 39& 20& 4& & & & & & \\\hline
S(4,3) &30& 86& 58& 15& 2& 1& & & & \\\hline
S(4,4) &44& 166& 146& 52& 12& 4& & & & \\\hline
S(4,3,2,2) &75& 391& 524& 176& 6& 1& & & & \\\hline
S(5,2) &28& 83& 72& 32& 4& 1& & & & \\\hline
S(6,2) &40& 157& 182& 95& 28& 4& & & & \\\hline
S(5,3) &42& 166& 174& 78& 16& 6& 1& & & \\\hline
S(6,3) &57& 290& 412& 210& 62& 14& 2& & & \\\hline
S(7,2) &55& 280& 432& 294& 130& 46& 4& 1& & \\\hline
S(5,5,5) &204& 2526& 10002& 10404& 5088& 1764& 444& 78& & \\\hline
S(6,5) &105& 813& 1678& 1136& 454& 149& 42& 12& 2& 1 \\\hline
\end{tabular}
\\

\vspace{0.1in}

\section{Universal Gr\"obner bases}

The Graver basis is a good approximation to the universal Gr\"obner basis, but they are not equal in general.  However,
extensive computations show evidence supporting the following conjecture:
\begin{conj}\label{UnivGrav}
$\mathcal U_\A = \mathcal Gr_\A$ for the defining matrix $\A$ of any rational normal scroll.
\end{conj}

Note that the defining ideal of $S:=S(n_1-1,\dots,n_c-1)$ is contained in the defining ideal of the scroll
$$S(n_1-1,\dots,n_c-1,\underbrace{1,\dots,1}_{l \mbox{ terms}}) $$    for any $l$.
Define $S'$ to be any such scroll, where $l$ is chosen so that the inequality
\begin{align*}
c+l + 3 > 2(n_P+n_Q-2-j_0)
\end{align*}
is satisfied, where $n_P+n_Q-2-j_0$ is the degree bound for the scroll $S'$ from  Theorem \ref{degGraverS}.
This puts a restriction on the size of the support of any primitive binomial.
Let $f\in\mathcal Gr_{\A}$.  Then $f\in I_{\A'}$ where $\A':=\A_{S'}$.  The primitivity of $f$ implies $f\in\mathcal Gr_{\A'}$.
If the conjecture is true for the scroll $S'$, then $f$ lies in the universal Gr\"obner basis of the ideal $I_{\A'}$,
and hence in the universal Gr\"obner basis of $I_\A$. \\
Therefore, to prove this conjecture, it suffices to prove a weaker one:
\begin{conj}
$\mathcal U_\A = \mathcal Gr_\A$ for rational normal scrolls of sufficiently high dimension.
\end{conj}

Recently, Hemmecke and Nairn  in \cite{He} stated that if the universal Gr\"obner basis and Graver basis of $I_\A$ coincide,
then the Gr\"obner and Graver complexities of $\A$ are equal.  We plan to study the higher Lawrence configurations of the rational normal scrolls.\\

Next, we consider state polytopes of rational normal scrolls.
Knowing a universal Gr\"obner basis of $I_\A$ is equivalent to knowing its {\emph{state polytope}} (\cite{St}).  It is defined to be any polytope whose normal
fan coincides with the {\emph{Gr\"obner fan}} of the ideal. The cones of the Gr\"obner fan correspond to the reduced Gr\"obner bases $\G_\<$ of $I_\A$. In
addition, the Gr\"obner fan is a refinement of the secondary fan $\mathcal N(\Sigma(\A))$, which classifies equivalence classes of lifting functions
giving a particular regular triangulation of the point configuration $\A$.
\begin{theorem}
The dimension of the state polytope of a rational normal scroll is one less then the degree of the scroll:  $$\dim \state(I_{S(n_1-1,\dots,n_c-1)}) = n_1+\dots
+n_c-c -1.$$
\end{theorem}

\begin{proof}
Eliminating variables results in taking faces of the state polytope. % this is ex.4 on p.16 of the book!!
Thus the state polytope for the scroll $S(n_1-1)$ is a face of
that of $S(n_1-1,1)$, which in turn is a face of the state polytope of $S(n_1-1,2)$, etc. so that each time we add a column to the parametrization matrix
$\A_S$, the dimension of the state polytope grows by at least one.
The ideal of the scroll $S(1,\dots,1)$ is just the ideal of $2$-minors of a generic $2\times c$ matrix.  The minors form a universal Gr\"obner basis for the ideal
which is a reduced Gr\"obner basis of the ideal with respect to every term order.  Hence, the state polytope is a Minkowski sum of the Newton polytopes of the
minors (Cor. 2.9. in \cite{St}), a permutohedron $\Pi_{2,c}$ (\cite{MaxMinor1},\cite{MaxMinor2}). Its dimension is $c-1$.

By induction, $$\dim \state (S(n_1-1,\dots,n_c-1)) \geq n_1-2 + n_2-1 + \dots + n_c-1 = \sum
n_i -c- 1.$$

On the other hand, the ideal of the scroll is homogeneous with respect to the grading given by all the rows of $\A_S$.  There are $c+1$ independent rows,
thus the vertices of the state polytope lie in $c+1$ hyperplanes, and the claim follows.
\end{proof}

Let us conclude with an example.
\begin{eg}
Let $S$ be the scroll $S(5,6)$.  %%%  blue, then red!
Its defining ideal $I_S$ is the ideal of $2$-minors of the matrix
\begin{align*}
M:=\mtx \textcolor{blue}{x_{1}} &\dots &\textcolor{blue}{x_{5}} &\textcolor{red}{y_{1}} &\dots &\textcolor{red}{y_{6}} \\
        \textcolor{blue}{x_{2}} &\dots &\textcolor{blue}{x_{6}} &\textcolor{red}{y_{2}} &\dots &\textcolor{red}{y_{7}} \mtxend .
\end{align*}
The matrix $\A$ providing the parametrization of the scroll is
\begin{align*}
\A=\mtx \textcolor{blue}{1}&\textcolor{blue}{1}&\dots &\textcolor{blue}{1}&\textcolor{red}{1}&\dots &\textcolor{red}{1} \\
        \textcolor{blue}{0}&\textcolor{blue}{0}&\dots &\textcolor{blue}{0}&\textcolor{red}{1}&\dots &\textcolor{red}{1} \\
        \textcolor{blue}{1}&\textcolor{blue}{2}&\dots &\textcolor{blue}{6}&\textcolor{red}{1}&\dots &\textcolor{red}{7} \mtxend .
\end{align*}

 The number and degrees of elements in the universal Gr\"obner basis of the ideal $I_\A$ can be found in the Table of degrees.
The primitive colored partition identity of maximal degree is
$$ \textcolor{blue}{1_1}+\textcolor{blue}{1_1}+\textcolor{blue}{1_1}+\textcolor{blue}{1_1}+\textcolor{blue}{1_1}+\textcolor{blue}{1_1}
+ \textcolor{red}{7_2}+\textcolor{red}{7_2}+\textcolor{red}{7_2}+\textcolor{red}{7_2}+\textcolor{red}{7_2}
= \textcolor{red}{1_2}+\textcolor{red}{1_2}+\textcolor{red}{1_2}+\textcolor{red}{1_2}+\textcolor{red}{1_2}
+ \textcolor{blue}{6_1}+\textcolor{blue}{6_1}+\textcolor{blue}{6_1}+\textcolor{blue}{6_1}+\textcolor{blue}{6_1}+\textcolor{blue}{6_1}.$$
The corresponding binomial in the ideal $I_\A$ is
$$ \textcolor{blue}{x_1}^6 \textcolor{red}{y_7}^5 = \textcolor{red}{y_1}^5 \textcolor{blue}{x_6}^6 .$$
The state polytope of the ideal $I_\A$ is $10$-dimensional.

There exist primitive elements that are not circuits. In fact, using \cite{4ti2}, we can see that there is a circuit in every degree from $2$ to $11$ except degree $10$,
but the number of  circuits in each degree is considerably smaller then the number of primitive binomials.
\end{eg}

\end{document}